\documentclass{article}
\usepackage{spconf,amsmath,graphicx}

\usepackage{amsthm,amsmath,amssymb}
\usepackage{graphicx}
\usepackage{amsfonts}
\usepackage{color}
\usepackage{hyperref}
\usepackage{bm}

\usepackage{algorithm}

\usepackage{cite}
\usepackage{textcomp}
\usepackage{xcolor}
\def\BibTeX{{\rm B\kern-.05em{\sc i\kern-.025em b}\kern-.08em
    T\kern-.1667em\lower.7ex\hbox{E}\kern-.125emX}}

\usepackage{balance}

 \newcounter{thm}
  \newcounter{re}
 \newtheorem{Proposition}[thm]{Proposition}

 \newtheorem{theorem}{Theorem}

\newtheorem{lemma}{Lemma}

\newcommand{\cS}{\mathcal{S}}
\newcommand{\cT}{\mathcal{T}_{\rm C}}

\title{First-order optimal sequential subspace change-point detection}
\name{Liyan Xie\sthanks{Work supported by the US National Science Foundation under Grant CCF\,1650913, CMMI\,1538746, and CCF\,1442635.} \hskip2cm George V. Moustakides\sthanks{Work supported by the US National Science Foundation under Grant CIF\,1513373, through Rutgers University.} \hskip2cm Yao Xie$^*$}
\address{$^*$Georgia Institute of Technology, School of Industrial and Systems Engineering,	Atlanta, GA, USA.\\
$^\dagger$Rutgers University, Department of Computer Science, New Brunswick, NJ, USA.}

\begin{document}

%\setlength{\abovedisplayskip}{6pt}
%\setlength{\belowdisplayskip}{6pt}
%\ninept
%
\maketitle
\begin{abstract}
We consider the sequential change-point detection problem of detecting changes that are characterized by a subspace structure. Such changes are frequent in high-dimensional streaming data altering the form of the corresponding covariance matrix. In this work we present a Subspace-CUSUM procedure and demonstrate its first-order asymptotic optimality properties for the case where the subspace structure is unknown and needs to be simultaneously estimated. To achieve this goal we develop a suitable analytical methodology that includes a proper parameter optimization for the proposed detection scheme. Numerical simulations corroborate our theoretical findings.
\end{abstract}

\vspace{-0.2in}
\section{Introduction}
\vspace{-0.15in}

Detecting changes in the distribution of high-dimensional streaming data is a fundamental problem in various applications such as swarm behavior monitoring \cite{berger2016classifying}, sensor networks, and seismic event detection. In various scenarios, the change can be represented as a linear subspace which is captured through a change in the covariance structure.

Given a sequence of samples $x_1, x_2, \ldots, x_t$, $t = 1,2,\ldots$, where $x_t \in \mathbb{R}^k$ and $k$ is the signal dimension, there may be a change-point time $\tau$ where the distribution of the data stream changes. Our goal is to detect this change as quickly as possible using on-line techniques.
We are particularly interested in the structured change that occurs in the signal covariance. We study two related settings, the {\it emerging subspace}: meaning that the change is a subspace emerging from a noisy background, and the {\it switching subspace}: meaning that the change is a switch in the direction of the subspace. The emerging subspace problem can arise from coherent weak signal detection from seismic sensor arrays, and the switching subspace detection can be used for principal component analysis for streaming data. In these settings, the change can be shown to be equivalent to a low-rank component added to the original covariance matrix.  

Classical approaches to covariance change detection usually consider generic settings without assuming any structure. 
The CUSUM statistics can be derived if the pre-change and post-change distributions are known. 
For the multivariate case, the Hotelling $T^2$ control chart is the traditional way to detect the covariance changes. The determinant of the sample covariance matrix was also used in \cite{alt2004multivariate} to detect change of the determinant of the covariance matrix. A multivariate CUSUM based on likelihood functions of multi-variate Gaussian is studied in \cite{healy1987note}  but it only considers the covariance change from $\Sigma$ to $c\Sigma$ for a constant $c$. 
Offline change detection of covariance change from $\Sigma_1$ to $\Sigma_2$ is studied in \cite{chen2004statistical} using the Schwarz information criterion \cite{schwarz1978estimating}, where the change-point location must satisfy certain regularity condition to ensure the existence of the maximum likelihood estimator. Recently, \cite{arias2012detection} studies the hypothesis test to detect a shift in the off-diagonal sub-matrix planted in the covariance matrix using the likelihood ratios.

In this paper, we propose the Subspace-CUSUM procedure by combing the CUSUM statistic with subspace estimation and proper parameter optimization. We prove that the resulting detector is first-order asymptotically optimal in the sense that the ratio of its expected detection delay with the corresponding of the optimum CUSUM (that has complete knowledge of the pre- and post-change statistics) tends to 1 as the average run length tends to infinity.

The rest of this paper is organized as follows. Section \ref{sec:sub} details on the two problems of emerging and switching subspace. Section \ref{sec:CUSUM} presents the Subspace-CUSUM procedure. Section \ref{sec:opt} considers the asymptotic analysis of the proposed scheme along with parameter optimization and proof of first-order asymptotic optimality. Finally, in Section \ref{sec:num} we present simulation results that corroborate our theoretical findings.

\vspace{-0.2in}
\section{Subspace Change-point detection}\label{sec:sub}
\vspace{-0.18in}
Both settings, emerging and switching subspace, can be shown to be related to the so-called {\it spiked covariance matrix} \cite{johnstone2001distribution}. For simplicity, we consider the rank-one spiked covariance matrix problem, which is given by 
\[
\Sigma = \sigma^2I_k + \theta uu^\intercal,
\]
where $I_k$ denotes an identity matrix of size $k$;  $\theta$ the signal strength; $u \in \mathbb{R}^{k\times 1}$ represents a basis for the subspace $\|u\|=1$ and $\sigma^2$ the noise power. We can define the Signal-to-Noise Ratio (SNR) as $\rho = \theta/\sigma^2$. 

In the {\it emerging subspace} problem the sequentially observed data are as follows
\begin{equation}
\begin{array}{ll}
x_t  \stackrel{\text{iid}}{\sim} \mathcal{N}(0,\sigma^2I_k ), &t = 1,2,\ldots,\tau, \\
x_t \stackrel{\text{iid}}{\sim} \mathcal{N}(0,  \sigma^2I_k + \theta u u^\intercal),& t = \tau+1,\tau+2, \ldots
\end{array}
\label{eq:hypothesis1}
\end{equation}
where $\tau$ is the unknown change-point that we would like to detect as soon as possible. We assume that the subspace $u$ is unknown since it represents anomaly or new information. 

In the {\it switching subspace} problem the data satisfy
\begin{equation}
\begin{array}{ll}
x_t  \stackrel{\text{iid}}{\sim} \mathcal{N}(0,\sigma^2I_k + \theta u_1u_1^\intercal), &t = 1,2,\ldots,\tau, \\
x_t \stackrel{\text{iid}}{\sim} \mathcal{N}(0,  \sigma^2I_k + \theta u_2 u_2^\intercal),& t = \tau+1,\tau+2,\ldots
\end{array}
\label{eq:hypothesis2}
\end{equation}
where $u_1$ and $u_2$ are the pre- and post-change subspaces. We assume that $u_1$ is completely known since it describes the statistical behavior under nominal conditions while $u_2$ is considered unknown since, as before, it expresses an anomaly.

The switching subspace problem \eqref{eq:hypothesis2} can be easily reduced into the emerging subspace problem \eqref{eq:hypothesis1}. Indeed if we select any orthonormal matrix $Q \in \mathbb{R}^{(k-1)\times k}$ that satisfies
\[ Q u_1 = 0, \quad QQ^\intercal = I_{k-1},\]
and project the observed data onto the space that is orthogonal to $u_1$ namely
$ y_t = Qx_t \in \mathbb{R}^{k-1}$,
then $y_t$ is a zero-mean random vector with covariance matrix $\sigma^2I_{k-1}$ before the change and $\sigma^2I_{k-1} + \tilde{\theta} uu^\intercal$ after the change where $u = Qu_2/\|Qu_2\|$, 
and 
\[
\tilde{\theta} = \theta \left\Vert Qu_2 \right\Vert^2 = \theta [ 1- (u_1^\intercal u_2)^2].% = \theta \sin^2 \alpha,
\]
%where $\alpha$ is the angle between $u_1$ and $u_2$. 
The data in \eqref{eq:hypothesis2} under this transformation becomes
\begin{equation}
\begin{array}{ll}
y_t  \stackrel{\text{iid}}{\sim} \mathcal{N}(0, \sigma^2I_{k-1}), &t = 1,2,\ldots,\tau, \\
y_t \stackrel{\text{iid}}{\sim} \mathcal{N}(0,  \sigma^2I_{k-1}  +\tilde{ \theta} u u^\intercal), & t = \tau+1,\tau+2,\ldots
\end{array}
\label{eq:hypothesis3}
\end{equation}
which is the emerging subspace problem in \eqref{eq:hypothesis1}. We need however to emphasize that by projecting the observations onto a lower dimensional space we lose information, suggesting that the two versions of the problem \textit{are not equivalent}. In particular the optimum detector for the transformed data in \eqref{eq:hypothesis3} and the one of the original data in \eqref{eq:hypothesis2} \textit{do not coincide}. This can be easily verified by computing the corresponding CUSUM tests and their (optimum) performance. Despite this difference, it is clear that with this result we are going to present next, and by adopting the transformed version \eqref{eq:hypothesis3}, we offer a computationally simple method to solve the original problem \eqref{eq:hypothesis2}. Therefore, from now on, our analysis will focus solely on detecting $\tau$ with the ccorresponding observations following the model depicted in \eqref{eq:hypothesis1}.

\vspace{-0.15in}
\section{Subspace CUSUM}\label{sec:CUSUM}
\vspace{-0.15in}

The CUSUM test \cite{page1954continuous, siegmund2013sequential}, when the observations are i.i.d. before and after the change, is known to be exactly optimum \cite{moustakides1986optimal} in the sense that it solves a very well defined constrained optimization problem introduced in \cite{lorden1971procedures}. If $f_\infty(x),f_0(x)$ denote the pre- and post-change probability density function (pdf) of the observations respectively then the CUSUM statistic $S_t$ and the corresponding CUSUM stopping time $T_{\text{C}}$ are defined \cite{moustakides1986optimal} as follows
\begin{equation}\label{exacTUSUM}
S_t = (S_{t-1})^+ +  \log\frac{f_0(x_t)}{f_\infty(x_t)},~~T_{\text{C}}=\inf\{t>0: S_t\geq b\},
\end{equation}
where $(x)^+ = \max\{x, 0\}$ and $b$ denotes a constant threshold. We must of course point out that application of CUSUM is only possible if we have exact knowledge of the pre- and post-change pdfs.

For the data model depicted in \eqref{eq:hypothesis1} the log-likelihood ratio takes the special form
\[
\log\frac{f_0(x_t)}{f_\infty(x_t)} = \frac1{2\sigma^2}\frac{\rho}{1+\rho} \Big\{(u^\intercal x_t)^2- \sigma^2 \!\left(\!1+\frac{1}{\rho}\right)\!\log(1+\rho)\Big\}.
\]
The multiplicative factor $\rho / [2\sigma^2(1+\rho) ]> 0$ can be omitted since it only performs a constant scaling of the test statistic. We can therefore define the CUSUM test statistic using the following recursion
\begin{equation}
S_t = (S_{t-1})^+ + (u^\intercal x_t)^2 - \sigma^2\left(1+\frac{1}{\rho}\right)\log(1+\rho).
\label{cusum_recur1}
\end{equation}
Using a simple argument based on Jensen's inequality, we can claim that the increment in \eqref{cusum_recur1} has a negative average under the nominal measure and a positive average under the alternative. Due to this property, the CUSUM statistic $S_t$ oscillates near $0$ before the change, and increases with a linear trend after the change.

Since in our case we assume that the vector $u$ is unknown we propose the following alternative to \eqref{cusum_recur1} with $u$ replaced by any estimate $\hat{u}_t$
\begin{equation}
\cS_t = (\cS_{t-1})^{+} + (\hat{u}_t^\intercal x_t)^2 - d.
\label{sscusum_update}
\end{equation}
Quantity $d$ is a constant that we would like to select properly so that the increment of $\cS_t$ mimic the main property of the increment of the CUSUM statistic $S_t$, that is, have a negative mean under nominal and a positive mean under the alternative probability measure. This will require
\begin{equation}\label{eq:d_interval}
\mathbb{E}_\infty[(\hat{u}_t^\intercal x_t)^2] < d < \mathbb{E}_0[(\hat{u}_t^\intercal x_t)^2].
\end{equation} 
The \textit{proposed} CUSUM-like stopping time is then defined as
\begin{equation}
\cT = \inf\{t>0: \cS_t \geq b\}.
\label{cusum_procedure}
\end{equation}

To be able to apply \eqref{sscusum_update} we need to specify $d$ and of course the estimate $\hat{u}_t$. Regarding the latter we propose a sliding window of size $w$ and form the sample covariance matrix
\[
\Sigma_t = \textstyle\sum_{i = t+1}^{t+w} x_i x_i^\intercal,
\]
using the observations $\{x_{t+1},\ldots,x_{t+w}\}$ that lie in \textit{the future} of $t$. Then $\hat{u}_t$ is simply the unit-norm eigenvector corresponding to the largest eigenvalue of $\Sigma_t$. The usage of observations from the future might seem somewhat awkward but it is always possible by properly delaying the data. The main advantage of this idea is that it provides estimates $\hat{u}_t$ that are \textit{independent} from $x_t$. Of course employing observations from times after $t$ affects the actual performance of our scheme. In particular, if with \eqref{cusum_procedure} we stop at time $\cT=t$ this implies that we used data from times up to $t+w$ and, consequently, $t+w$ is the true time we stop and not $t$.

The independence between $\hat{u}_t$ and $x_t$ allows for the simple computation of the two expectations in \eqref{eq:d_interval}. However, for this computation to be possible, especially under the alternative regime, it is necessary to be able to describe the statistical behavior of our estimate $\hat{u}_t$. We will assume that the window size $w$ is sufficiently large so that Central Limit Theorem type approximations are possible for $\hat{u}_t$ and we will consider that $\hat{u}_t$ is actually Gaussian with mean $u$ (the correct vector) and (error) covariance matrix that can be specified, analytically, of being size $1/w$ \cite{anderson1963asymptotic, paul2007asymptotics}. Explicit formulas will be given in the Appendix.

\vspace{-0.2cm}
\begin{lemma}\label{lem:d_bound} Adopting the Gaussian approximation for $\hat{u}_t$ we have the following two mean values under the pre- and post-change regime:
\begin{equation*}
\begin{aligned}
\mathbb{E}_\infty[(\hat{u}_t^\intercal x_t)^2] &=  \sigma^2,\\
\mathbb{E}_0[(\hat{u}_t^\intercal x_t)^2] &  = \sigma^2(1+\rho)\left[ 1 - \frac{k-1}{w\rho} \right].
\end{aligned}
\end{equation*}
\end{lemma}
\vspace{-0.2in}
\begin{proof} The proof is given in the Appendix.
\end{proof}
\vspace{-0.1in}
Lemma \ref{lem:d_bound} also suggests that the window size $w$ and the drift $d$ must satisfy
\begin{equation} \label{d_interval}
 \sigma^2 < d < \sigma^2(1+\rho)\left( 1 - \frac{k-1}{w\rho} \right).
\end{equation}
Necessary condition for this to be true is that $w > (k-1)(1+\rho)/\rho^2$. Actually this constraint is required for the Gaussian approximation to make sense. But in order for the approximation to be efficient we, in fact, need $w$ to be significantly larger than the lower bound. We can see that when the SNR is high ($\rho\gg1$) then with relatively small window size we can obtain efficient estimates. When on the other hand SNR is low ($\rho\ll1$) then far larger window sizes are necessary to guarantee validity of the Gaussian approximation.

\vspace{-0.1in}
\section{Asymptotic analysis}\label{sec:opt}
\vspace{-0.15in}
In this section we will provide performance estimates for the optimum CUSUM test (that has all the information regarding the data) and the Subspace-CUSUM test proposed in the previous section. This will allow for the optimum design of the two parameters $w,d$ and for demonstrating that the resulting detector is asymptotically optimum.

In sequential change detection there are two quantities that play vital role in the performance of a detector: a)~the average run length (ARL) and b)~the expected detection delay (EDD). ARL measures the average period between false alarms while EDD the (worst-case) average detection delay. It is known that CUSUM minimizes the latter while keeps the former above a prescribed level. Let us first compute these two quantities for the case of CUSUM given in \eqref{exacTUSUM}.
\vspace{-0.2in}
\subsection{Asymptotic performance} \label{sec:metric}
\vspace{-0.1in}
From \cite[Pages 396--397]{tartakovsky2014sequential} we have that the test depicted in \eqref{exacTUSUM} has the following performance
\begin{equation}\label{arl&edd}
\mathbb{E}_\infty[T_{\text{C}}] = \frac{e^b}{\mathsf{K}}\big(1+o(1)\big), ~~~\mathbb{E}_0[T_{\text{C}}] = \frac{b}{\mathsf{I}_0}\big(1+o(1)\big),
\end{equation}
where $b$ is the constant threshold; $\mathsf{K}$ is of the order of a constant with its exact value being unimportant for the asymptotic analysis; finally $\mathsf{I}_0$ is the Kullback-Leibler information number $\mathsf{I}_0 = \mathbb{E}_0\{\log \left[ f_0(x)/f_\infty(x)\right]\}$. We recall that the worst-case average detection delay in CUSUM is equal to $\mathbb{E}_0[T_{\text{C}}]$. This is the reason we consider the computation of this quantity. If now, we impose the constraint that the ARL is equal to $\gamma>1$ and for the asymptotic analysis that $\gamma\to\infty$, then we can compute the threshold $b$ that can achieve this false alarm performance namely $b=(\log\gamma)\big(1+o(1)\big)$. Substituting this value of the threshold in EDD we obtain
\begin{equation}
\mathbb{E}_0[T_{\text{C}}]=\frac{\log\gamma}{\mathsf{I}_0}\big(1+o(1)\big).
\label{eq:mbofla}
\end{equation}
Applying this formula in our problem we end up with the following optimum performance
\begin{equation}
\mathbb{E}_0[T_{\text{C}}]=\frac{2\log\gamma}{\rho-\log(1+\rho)}\big(1+o(1)\big).
\label{eq:mbifla}
\end{equation}

For the performance computation of Subspace-CUSUM, since the increment $(\hat{u}_t^\intercal x)^2-d$ in \eqref{sscusum_update} is not a log-likelihood, we cannot use \eqref{eq:mbofla} directly.
To compute the ARL of $\cT$ we first find $\delta_\infty>0$ from the solution of the equation
\begin{equation} \label{delta_constratint}
 \mathbb{E}_\infty[e^{\delta_\infty[(\hat{u}_t^\intercal x_t)^2 - d]}] = 1
 \end{equation}
and then we note that $\delta_\infty[(\hat{u}_t^\intercal x)^2-d]$ is the log-likelihood ratio between the two pdfs
$\tilde{f}_0=\exp\{\delta_\infty[(\hat{u}_t^\intercal x)^2-d]\} f_\infty$ and $f_\infty$. This allows us to compute the threshold $b$ asymptotically as $b=(\log\gamma)\big(1+o(1)\big)/\delta_\infty$ after assuming that $w=o(\log\gamma)$. Similarly we can find a $\delta_0>0$ and define $\tilde{f}_\infty=\exp\{-\delta_0[(\hat{u}_t^\intercal x_t)^2 - d]\}f_0$ so that $\delta_0[(\hat{u}_t^\intercal x_t)^2 - d]$ is the log-likelihood ratio between $f_0$ and $\tilde{f}_\infty$ leading to $\mathbb{E}_0[\cT]=b\big(1+o(1)\big)/(\mathbb{E}_0[(\hat{u}_t^\intercal x_t)^2]-d)$ with the dependence on $\delta_0$ being in the $o(1)$ term. Substituting $b$ we obtain
\begin{equation}
\mathbb{E}_0[\cT]=\frac{\log\gamma}{\delta_\infty\big(\mathbb{E}_0[(\hat{u}_t^\intercal x_t)^2]-d\big)}\big(1+o(1)\big)+w,
\end{equation}
where the last term $w$ is added because we use data from the future of $t$ as we explained before.
Parameter $\delta_\infty$, defined in \eqref{delta_constratint}, is directly related to $d$. We show in the Appendix that $d$ can be expressed in terms of $\delta_\infty$ as follows
\begin{equation} 
d = -\frac1{2\delta_\infty}\log(1-2\sigma^2\delta_\infty). 
\label{d_delta}
\end{equation}
After using Lemma\,\ref{lem:d_bound} and \eqref{d_delta} we obtain the following expression for the EDD:
%\clearpage
\begin{equation}\label{edd_star}
\mathbb{E}_0[ \cT ]=\textstyle
\frac{ \log\gamma(1+o(1))}{\sigma^2\delta_\infty(1+\rho)\left( 1 - \frac{k-1}{w\rho} \right)  +\frac12\log(1-2\sigma^2\delta_\infty) }+w.
\end{equation}
\vspace{-0.3in}
\subsection{Parameter optimization and asymptotic optimality}
\vspace{-0.1in}
Note that in the previous equation we have two parameters $\delta_\infty$ and $w$ and the goal is to select them so as to minimize the EDD. Therefore if we first fix the window size $w$ we can minimize over $\delta_\infty$ (which is equivalent to minimizing with respect to the drift $d$). We observe that the denominator is a concave function of $\delta_\infty$ therefore it exhibits a single maximum. The optimum $\delta_\infty$ can be computed by taking the derivative and equating to 0 which leads to a particular ${\delta}_\infty$. Substituting this optimal value we obtain the following minimum EDD:
\begin{multline}\label{edd2}
\mathbb{E}_0[\cT]=
\textstyle
\frac{2 \log\gamma(1+o(1))}{(1+\rho)\left( 1 - \frac{k-1}{w\rho} \right) -1 - \log \left[ (1+\rho)\left( 1 - \frac{k-1}{w\rho} \right)\right]}+w.
\end{multline} 
Equ.~\eqref{edd2} involves only the target ARL level $\gamma$ and the window size $w$. If we keep $w$ constant it is easy to verify that the ratio of the EDD of the proposed scheme over the EDD of the optimum CUSUM tends, as $\gamma\to\infty$, to a quantity which is strictly greater than 1. In order to make this ratio tend to 1 and therefore establish asymptotic optimality we need to select the window size $w$ as a function of $\gamma$. Actually we can perform this selection optimally by minimizing \eqref{edd2} with respect to $w$ for given $\gamma$. The following proposition identifies the optimum window size.
\vspace{-0.05in}
\begin{Proposition}\label{optimal_d_w}
For each ARL level $\gamma$, the optimal window size that minimizes the corresponding EDD is given by
$$
w^* = \textstyle
\sqrt{\log\gamma}\cdot\frac{ \sqrt{2(k-1)}}{\rho - \log(1+\rho)}\big(1+o(1)\big),
$$
resulting in an optimal drift
$$
d^* =\textstyle 
\frac{ \sigma^2 (1+\rho)\left( 1 - \frac{k-1}{w^*\rho} \right) }{(1+\rho)\left( 1 - \frac{k-1}{w^*\rho} \right)-1}  \log \left[ (1+\rho)\left( 1 - \frac{k-1}{w^*\rho} \right)\right].
$$
\end{Proposition}
\vspace{-0.1in}
Using these optimal parameter values it is straightforward to establish that the corresponding Subspace-CUSUM is first-order asymptotically optimum. This is summarized in our next theorem.
\vspace{-0.05in}
\begin{theorem}\label{first_order_opt}
As the ARL level $\gamma\to\infty$, the corresponding EDD of the Subspace-CUSUM procedure $\cT$ with the two parameters $d$ and $w$ optimized as above satisfies
\begin{equation} \label{edd_minimal}
\lim_{\gamma\to\infty}\frac{\mathbb{E}_0[\cT]}{\mathbb{E}_0[T_{\rm C}]}=1.
\end{equation}
\end{theorem}
\vspace{-0.2in}
\begin{proof}
As we pointed out, the proof is straightforward. Indeed if we substitute the optimum $d$ and $w$ and then take the ratio with respect to the optimum CUSUM performance depicted in \eqref{eq:mbifla} we obtain
\vspace{-0.1in}
$$
\vspace{-0.1in}
\frac{\mathbb{E}_0[\cT]}{\mathbb{E}_0[T_{\text{C}}]}=1 + \sqrt{\frac{k-1}{2\log\gamma}} + o(1)\to 1,
$$
which proves the desired limit. Even though the ratio tends to 1, we note that $\mathbb{E}_0[\cT]-\mathbb{E}_0[T_{\text{C}}]=\Theta(\sqrt{\log\gamma})\to\infty$. This is corroborated by our simulations (see Fig.\,\ref{fig:eddcompare}, red curve).
\end{proof}

\vspace{-0.25in}
\section{Numerical examples}\label{sec:num}
\vspace{-0.15in}
We present simulations to illustrate the satisfactory performance of Subspace-CUSUM. For comparison, we consider two other detection procedures: one uses the largest eigenvalue of the sample covariance matrix $\Sigma_{t}$ as the test statistic while the other is the exact CUSUM assuming all parameters are known (ideal but unrealistic case).

The threshold for each detection procedure is determined through Monte-Carlo simulation so they all have the same ARL. Fig.\,\ref{fig:eddcompare} depicts the EDD versus ARL with the latter under a logarithmic scale. Parameters are selected as follows: $k=5$, $\theta=1$, $\sigma^2 = 1$ and window length $w = 20$. Exact CUSUM (black) is compared against Subspace-CUSUM (green) and largest eigenvalue scheme (blue). We see that Subspace-CUSUM has much smaller EDD than the largest eigenvalue procedure while Subspace-CUSUM with optimized window size $w$ (red) is uniformly more efficient.
\begin{figure}[b!]
\vspace{-0.5cm}
\centerline{
\includegraphics[width = 0.34\textwidth]{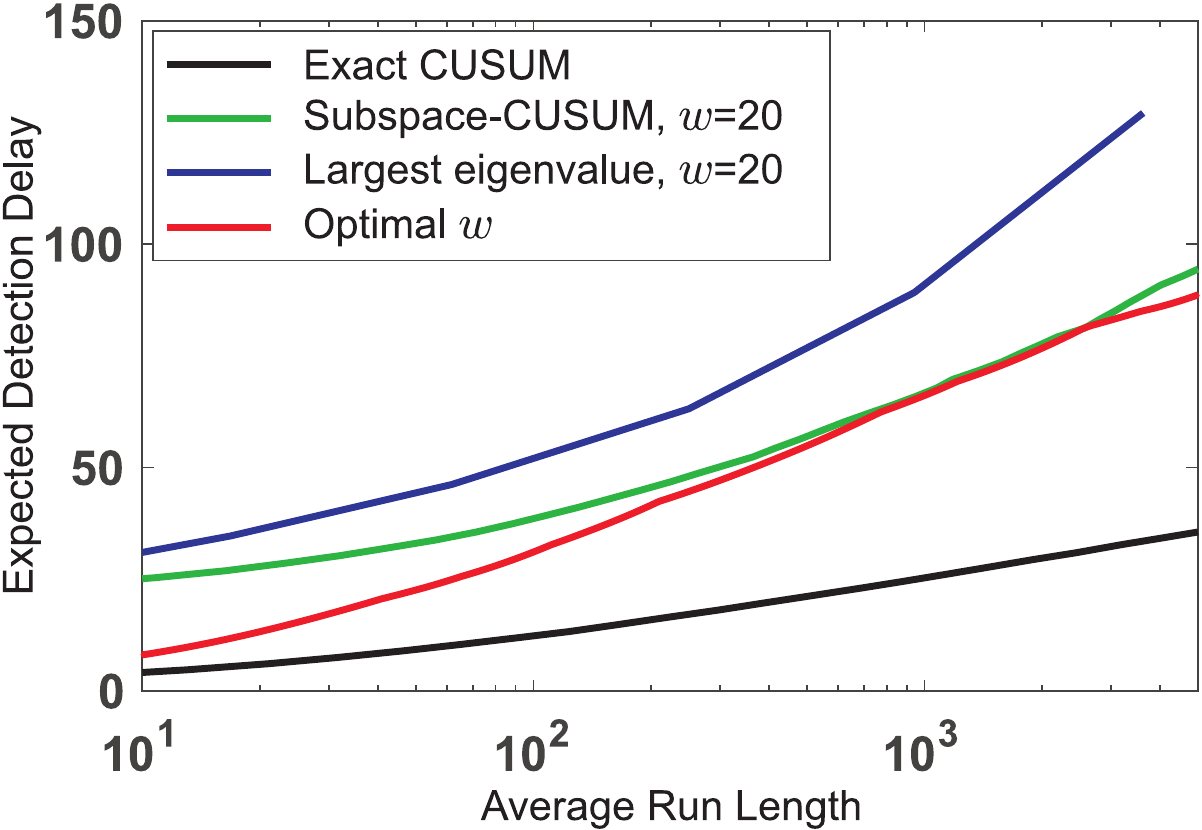} 
}
\vspace{-.1in}
\caption{Comparison of the largest eigenvalue procedure and CUSUM procedures.}
\label{fig:eddcompare}
\vspace{0.4cm}
\centerline{\includegraphics[width = 0.24\textwidth]{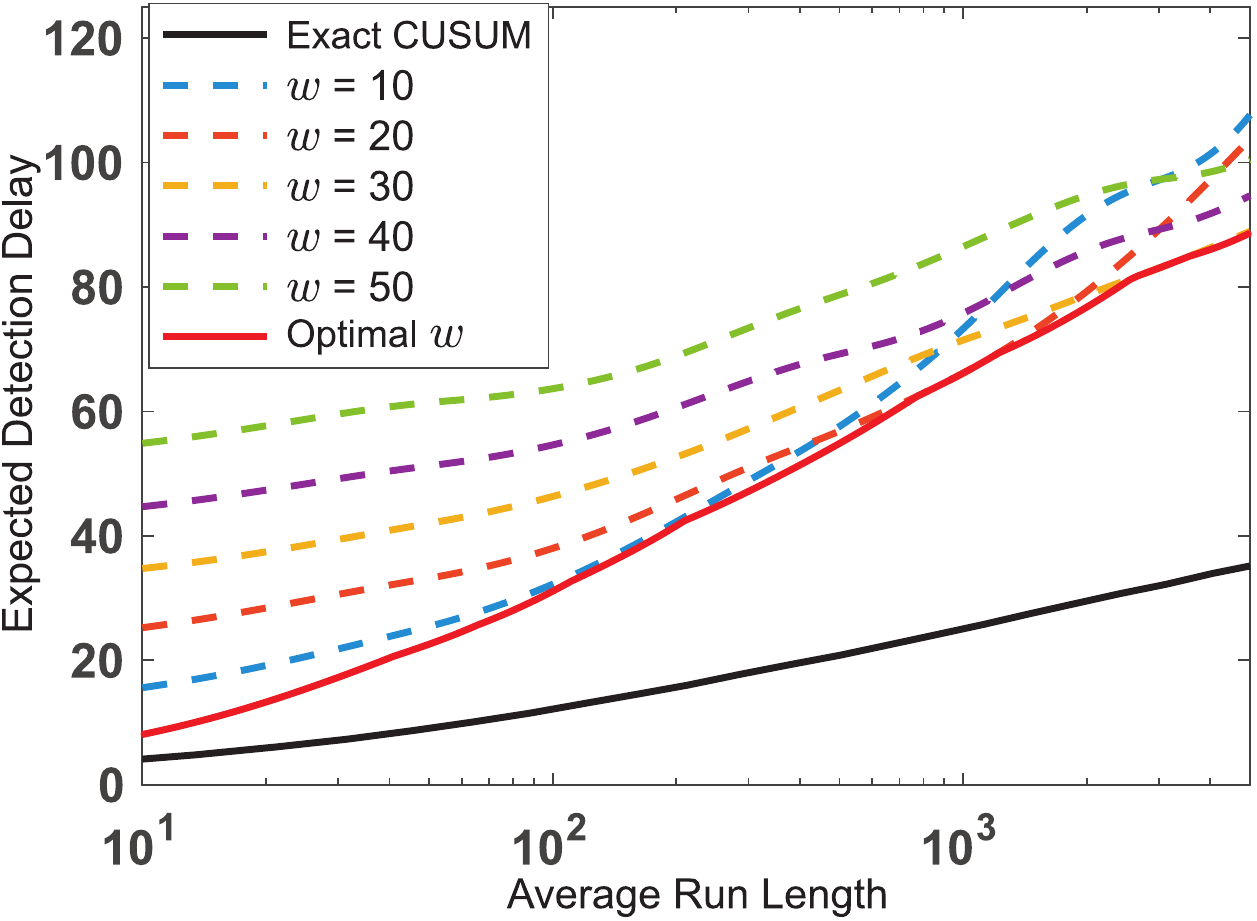} 
\hfill\includegraphics[width = 0.24\textwidth]{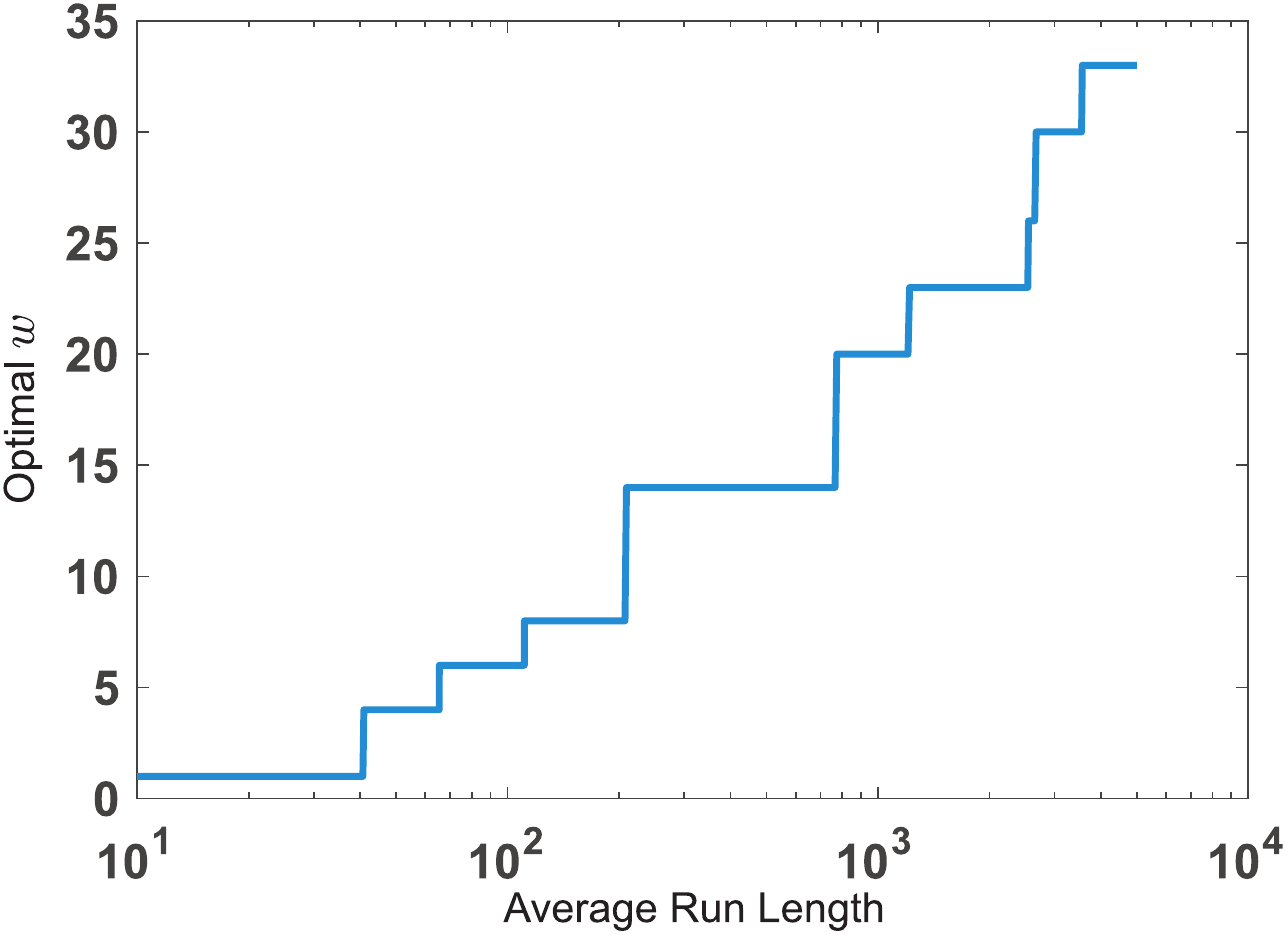}}
\vspace{-.1in}
\caption{Left: Minimal EDD (red) among window sizes $w$ from 1 to 50; Right: Corresponding optimal window size $w$.}
\label{fig:optimalw}
\vspace{-.1in}
\end{figure}
We also consider EDD versus ARL for different $w$ and with numerically optimized $w$ so as to minimize the detection delay for each ARL level. The results appear in Fig.\,\ref{fig:optimalw}, which demonstrate that indeed the optimal $w$ increases with ARL.

%\section{Conclusion}

\clearpage

\bibliographystyle{IEEEbib}
\bibliography{CAMSAP2017}
\balance
\clearpage
\appendix

\section{Appendix}

\begin{proof}[Proof of Lemma \ref{lem:d_bound}]~
Using the independence between $\hat{u}_t$ and $x_t$ we can write
\begin{equation}
\mathbb{E}[(\hat{u}_t^\intercal x_t)^2]=
\mathbb{E}\big[\hat{u}_t^\intercal \mathbb{E}[x_tx_t^\intercal]\hat{u}_t\big].
\label{eq:A1}
\end{equation}
Consequently, under the nominal regime
$$
\mathbb{E}_\infty[(\hat{u}_t^\intercal x_t)^2]=
\mathbb{E}_\infty[\hat{u}_t^\intercal \sigma^2 I_k\hat{u}_t\big]=\sigma^2,
$$
with the last equality being true because $\hat{u}_t$ is of unit norm.

Under the alternative regime we are going to use Central Limit Theorem arguments \cite{anderson1963asymptotic, paul2007asymptotics} that describe the statistical behavior of the estimator. We have that
$$ 
\sqrt{w}(\omega_t - u) \rightarrow \mathcal{N}\left(0, \frac{1+\rho}{\rho^2}(I_k-uu^\intercal)\right)
$$
where the limit is in distribution as $w\to\infty$ and $\omega_t$ denotes the \textit{un-normalized} eigenvector. For large $w$ we can write $\omega_t=u+v_t$ where
$$
v_t\sim \mathcal{N}\left(0, \frac{1+\rho}{w\rho^2}(I_k -uu^\intercal)\right).
$$
Our estimator $\hat{u}_t$ is related to $\omega_t$ through the normalization process $\hat{u}_t=\omega_t/\|\omega_t\|$, and if we use this in \eqref{eq:A1} after recalling that under the alternative $\mathbb{E}_0[x_tx_t^\intercal]=\sigma^2(I_k+\rho uu^\intercal)$ and using repeatedly the fact that $u$ and $v_t$ are orthogonal, we have
\begin{multline*}
\mathbb{E}_0[(\hat{u}_t^\intercal x_t)^2]=
\sigma^2\mathbb{E}_0\big[\hat{u}_t^\intercal (I_k+\rho u u^\intercal)\hat{u}_t\big]\\
=\sigma^2(1+\rho\mathbb{E}_0[(\hat{u}_t^\intercal u)^2])
=\sigma^2\left(1+\rho\mathbb{E}_0\left[\frac{1}{1+\|v_t\|^2}\right]\right)\\
\approx\sigma^2\left(1+\rho\mathbb{E}_0\left[1-\|v_t\|^2\right]\right)
=\sigma^2(1+\rho)\left(1-\frac{k-1}{w\rho}\right).
\end{multline*}
For the approximate equality we used the fact that to a first order approximation we can write $1/(1+\|v_t\|^2)\approx 1-\|v_t\|^2$ because $\|v_t\|^2$ is of the order of $1/w$ while the approximation error is of higher order. This completes the proof.
\end{proof}
 
\begin{proof}[Proof of Proposition \ref{optimal_d_w}]

Let us first evaluate the expectation in \eqref{delta_constratint} to demonstrate the relationship between $\delta_\infty$ and $d$ depicted in \eqref{d_delta}. Using standard computations involving Gaussian random vectors we can write
\begin{multline*}
\mathbb{E}_\infty[e^{\delta_\infty[(\hat{u}_t^\intercal x_t)^2 - d]}]
= e^{-\delta_\infty d}\mathbb{E}_\infty \left[ \mathbb{E}_\infty [e^{\delta_\infty(\hat{u}_t^\intercal x_t)^2} | \hat{u}_t ] \right]\\
=e^{-\delta_\infty d} \mathbb{E}_\infty \left[ \int e^{\delta_\infty x_t^\intercal( \hat{u}_t \hat{u}_t^\intercal) x_t} \cdot \frac{ e^{-x_t^\intercal x_t /(2\sigma^2)}}{ \sqrt{ (2\pi)^k\sigma^{2k} }} dx_t \right]\\
=\frac{e^{-\delta_\infty d} }{\sqrt{1-2\sigma^2\delta_\infty }}.
\end{multline*}
To compute the integral we used the standard technique of ``completing the square'' in the exponent and with proper normalization generate an alternative Gaussian pdf which integrates to 1. The interesting observation is that the result of the integration does not actually depend on $\hat{u}_t$.

If we use the optimum value for $d$ in terms of $\delta_\infty$ then as we argued in the text we obtain for EDD the expression appearing in \eqref{edd_star}. We can now fix $w$ and optimize EDD with respect to $\delta_\infty$. This is a straightforward process since it amounts in maximizing the denominator. Taking the derivative and equating to 0 yields the optimum $\delta_\infty$
$$
 \delta^*_\infty  = \frac{1}{2\sigma^2} \left( 1- \frac1{ (1+\rho)\left( 1 - \frac{k-1}{w\rho} \right)} \right).
$$
Substituting this value in \eqref{edd_star} produces \eqref{edd2}.

The next step consists in minimizing \eqref{edd2} with respect to $w$. Again taking the derivative and equating to 0 we can show that the optimum window size is the $w^*$ depicted in Proposition \ref{optimal_d_w}.
\end{proof}

%\section{ARL, EDD}


\end{document}